\documentclass[12pt]{amsart}
\usepackage{amsfonts, amssymb, latexsym, hyperref, xcolor, tikz,tikz-cd,transparent}
\usepackage{ytableau}
\usepackage{dynkin-diagrams}
\usepackage[letterpaper,left=2.25cm,right=2.25cm,top=3cm,bottom=3cm,headsep=1cm]{geometry}
\usetikzlibrary{topaths, calc, shapes,arrows,fit,positioning}
\usepackage{graphicx}
\usepackage{xparse}
\usepackage{mathtools}
\usepackage{hyperref}
\usepackage{subcaption}

\usepackage{todonotes}

\hypersetup{colorlinks, linkcolor={blue!50!purple},
citecolor={green!60!blue}, urlcolor={blue!80!black}}

\newcommand{\C}{\mathbb{C}}
\newcommand{\Z}{\mathbb{Z}}

\newcommand{\init}{{\sf init}}

\newcommand{\Mat}{{\sf Mat}}

\newcommand{\dyckpath}[4]{
    \fill[white!25]  (#1) rectangle +(#2,#3);
    \fill[fill=white]
    (#1)
    \foreach \dir in {#4}{
        \ifnum\dir=0
        -- ++(1,1)
        \else
        -- ++(1,-1)
        \fi
    } |- (#1);
    \draw[help lines] (#1) grid +(#2,#3);
    \coordinate (prev) at (#1);
    \foreach \dir in {#4}{
        \ifnum\dir=0
        \coordinate (dep) at (1,1);
        \else
        \coordinate (dep) at (1,-1);
        \fi
        \draw[line width=2pt,-stealth] (prev) -- ++(dep) coordinate (prev);
    };
}

    \newtheorem{theorem}{Theorem}[section]
    \newtheorem{lemma}[theorem]{Lemma}
    \newtheorem{prop}[theorem]{Proposition}
    \newtheorem{corollary}[theorem]{Corollary}
    \newtheorem*{thm*}{Theorem}
    
\theoremstyle{definition}
    \newtheorem{definition}[theorem]{Definition}
    \newtheorem{example}[theorem]{Example}
\theoremstyle{remark}
    \newtheorem{remark}[theorem]{Remark}
\numberwithin{equation}{section}

\title{Plane Partitions and Spin Adapted Quantum States}
\author{Abigail Price}
\address{Dept.~of Mathematics, U.~Illinois at Urbana-Champaign, Urbana, IL 61801, USA}
\email{price29@illinois.edu}
\author{Ada Stelzer}
\address{Dept.~of Mathematics, U.~Illinois at Urbana-Champaign, Urbana, IL 61801, USA}
\email{astelzer@illinois.edu}
\author{Svala Sverrisdóttir}
\address{Merton College, University of Oxford, OX1 4JD, UK}
\email{svala.sverrisdottir@merton.ox.ac.uk}
\date{\today}

\begin{document}

\begin{abstract}
   We describe an explicit basis for the $\operatorname{SU}(2)$-invariant space of the $2k$-th exterior power of a $2m$-dimensional complex vector space via the combinatorics of plane partitions. 
   In quantum chemistry, this is the space of spin adapted quantum states of an electronic system with $m$ spatial orbitals and $k$ electron pairs.
   We construct our basis by identifying the invariant space with an Artinian commutative ring called the \emph{excitation ring}. 
   We compute a Gröbner basis and enumerate its standard monomials via an explicit bijection to Dyck paths counted by the Narayana numbers.
\end{abstract}

\maketitle

\section{Introduction}

Fix positive integers $1 \le k \le m$ and consider the indeterminate $k \times (m - k)$ matrix $X = (X_{i,j})$. 
We write $\C[X]$ for the polynomial ring generated by the $k(m - k)$ variables $X_{i, j}$ under lexicographic term order (i.e., $X_{i,j} > X_{i',j'}$ if $i < i'$ or if $i = i'$ and $j < j'$).
Let $I_{m,k}\subseteq\C[X]$ be the homogeneous ideal generated by the following cubics:
\begin{equation}\label{eq:gens}
    f_{p, q, r}^{a, b, c} = \sum_{\sigma \in \mathfrak{S}_{\{a,b,c\}}} X_{p,\sigma(a)}X_{q,\sigma(b)} X_{r,\sigma(c)}, \text{ where }  p \le q \le r \le k \text{ and } a \le b \le c \le m - k.
\end{equation}
Here, $\mathfrak{S}_{\{a,b,c\}}$ denotes the set of permutations of the $3$-element multiset $\{a,b,c\}$.
Each cubic $f_{p, q, r}^{a, b, c}$ has $1, 2, 3$ or $6$ distinct terms, with positive integer coefficients summing to $6$. For example, all monomial generators are of the form 
$f_{p, p, p}^{a, b, c} = 6X_{p,a}X_{p,b}X_{p,c}$ or $f_{p, q, r}^{a, a, a} = 6X_{p,a}X_{q,a}X_{r,a}$.

\begin{example}
    Let $m = 4$ and $k = 2$. The ideal $I_{4,2}$ is generated by $16$ cubics of the form (\ref{eq:gens}). There are $12$ monomial generators in this case:
    \begin{align*}
        6X_{1,1}^3,\ 6X_{1,2}^3,&\ 6X_{2,1}^3,\ 6X_{2,2}^3,\\
        6X_{1,1}^2X_{1,2},\ 6X_{1,1}X_{1,2}^2,\  6X_{2,1}^2X_{2,2},\ 6X_{2,1}X_{2,2}^2,&\ 
    6X_{1,1}^2X_{2,1},\ 6X_{1,1}X_{2,1}^2,\  6X_{1,2}^2X_{2,2},\ 6X_{1,2}X_{2,2}^2.
    \end{align*}
    There are also $4$ binomial generators, presented here with their leading terms underlined:
    \begin{align*}
        &f_{1, 1, 2}^{1, 1, 2} = \underline{2X_{1,1}^2X_{2,2}} + 4X_{1,1}X_{1,2}X_{2,1}, \quad f_{1, 2, 2}^{1, 2, 2} = \underline{2X_{1,1}X_{2,2}^2} + 4X_{1,2}X_{2,1}X_{2,2},\\
    &f_{1, 1, 2}^{1, 2, 2} = \underline{4X_{1,1}X_{1,2}X_{2,2}} + 2X_{1,2}^2X_{2,1}, \quad
    f_{1, 2, 2}^{1, 1, 2} = \underline{4X_{1,1}X_{2,1}X_{2,2}} + 2X_{1,2}X_{2,1}^2.
    \end{align*}
    The quotient $\C[X]/I_{4,2}$ has Krull dimension $0$ and degree $20$.
\end{example}

Our main result determines the vector space dimension for $\C[X]/I_{m, k}$ by indexing the standard monomials of $I_{m,k}$ by certain Dyck paths:

\begin{theorem}\label{thm:degofex}
    The quotient ring $S_{m, k} := \C[X]/I_{m, k}$ is a vector space of dimension
    \[\dim_\C(S_{m, k}) = \frac{1}{m+1}\binom{m+1}{k+1}\binom{m+1}{k}.\]
\end{theorem}

Our motivation for Theorem~\ref{thm:degofex} comes from quantum chemistry. In that setting, the quotient $S_{m, k}$ is called the \emph{excitation ring}, a term we will use for $S_{m,k}$. Theorem~\ref{thm:degofex} yields an explicit description for the space of spin adapted electronic quantum states.

Our proof of Theorem~\ref{thm:degofex} is combinatorial.
In Section~\ref{sec:stdmono} we show that the defining generators of $I_{m, k}$ form a Gr\"obner basis via a representation theoretic argument. 
Our argument builds from the observation that the ideal $I_{m, k}$ has appeared previously in the literature as the ideal of \emph{$3\times 3$ generalized permanents} of a generic matrix, see \cite[Example 11.8.5]{BCRV}. 
We thereby obtain a vector space basis consisting of the \emph{standard monomials} for $S_{m, k}$, indexed by certain nonnegative integer matrices. 
Section~\ref{sec:bij} enumerates this basis by constructing a bijection between the standard monomials and Dyck paths of length $2m + 2$ with exactly $k$ valleys, which are enumerated by the \emph{Narayana number} $N(m+1, k+1) = \frac{1}{m+1}\binom{m+1}{k+1}\binom{m+1}{k}$.
Theorem~\ref{thm:degofex} follows immediately from this enumeration.

In Section \ref{se:comtoquant}, we explain the relevance of the excitation ring $S_{m,k}$ in the study of spin adapted electronic systems, see \cite[Section~2]{HelgakerJorgensenOlsen} and \cite{FaulstichSverrisdottir}. We identify $S_{m,k}$ as the invariant ring of an $\operatorname{SU}(2)$-action and show it is isomorphic to an Artinian commutative ring generated by \textit{excitation operators}, justifying our terminology. The proof of Theorem \ref{thm:degofex} provides an explicit model for the space of spin adapted quantum states via the combinatorics of plane partitions, enabling future study of this space using algebro-geometric methods.

\section{\texorpdfstring{Standard monomials for $S_{m, k}$}{Standard monomials for the excitation ring}}\label{sec:stdmono}

The objective of this section is to describe a vector space basis for the excitation ring $S_{m, k}$. This is accomplished via the theory of \emph{Gr\"obner bases}, which algorithmically constructs a vector space basis of \emph{standard monomials} for any ideal quotient $\C[X]/I$, see \cite{Buchberger}. To prove that the generators in~\eqref{eq:gens} form a Gr{\"o}bner basis for $I_{m,k}$, we make use of the representation theory of $\mathrm{GL}_n(\C)$. We briefly sketch the necessary background; see \cite{BCRV} for more details. 

Since $\mathrm{GL}_n(\C)$ is reductive, every finite-dimensional $\mathrm{GL}_n(\C)$-representation may be written as a direct sum of irreducible representations. The irreducible polynomial representations $V_{\lambda}(n)$ of $\mathrm{GL}_n(\C)$ are indexed by partitions $\lambda$ of length at most $n$, and their dimensions are found by counting \emph{semistandard Young tableaux}.

\begin{definition}
    A \emph{semistandard Young tableau} of partition \emph{shape} $\lambda$ with entries from $[n]$ is a filling of the Young diagram of $\lambda$ with numbers from $[n]:=\{1,\dots,n\}$ (with repetitions allowed), such that the entries are weakly increasing from left to right along rows and strictly increasing down columns from top to bottom. The set of all semistandard Young tableaux with shape $\lambda$ and entries from $[n]$ is denoted ${\sf SSYT}(\lambda, n)$. 
\end{definition}

\begin{example}
    Let $\lambda = (4,2,1)$. The filling $$\begin{ytableau}
        1 & 1 & 2 & 4\\
        2 & 3 \\
        5
    \end{ytableau}$$ is a semistandard Young tableau of shape $\lambda$ with entries from $[n]$ for any $n\geq 5$.
\end{example}

\begin{prop}[{\cite[pg. 119]{Fulton}}]
    The dimension of $V_\lambda(n)$ is $|{\sf SSYT}(\lambda,n)|$.
\end{prop}

The ring $\C[X]$ carries an action of $\mathrm{GL}_k(\C)\times \mathrm{GL}_{m-k}(\C)$ by 
\[
(A,B)\cdot f(X) = f(AXB^T), \; \forall (A,B)\in \mathrm{GL}_k(\C)\times \mathrm{GL}_{m-k}(\C), \; \forall f(X)\in \C[X].
\] Each graded piece $\C[X]_d$ of $\C[X]$, consisting of homogeneous polynomials of total degree $d$, is stable under this action. Since products of reductive groups are reductive, $\mathrm{GL}_k(\C)\times \mathrm{GL}_{m-k}(\C)$ is reductive and each $\C[X]_d$ decomposes into a direct sum of irreducible representations of $\mathrm{GL}_k(\C)\times \mathrm{GL}_{m-k}(\C)$. The decomposition of $\C[X]$ into irreducibles is well-known.

\begin{prop}[{Cauchy identity, \cite[Theorem~7.12.1]{ECII}}]
    As representations of $\mathrm{GL}_k(\C)\times \mathrm{GL}_{m-k}(\C)$, 
    \[
    \C[X]\cong \bigoplus_{\lambda:\ell(\lambda)\leq \min\{k, m-k\}}V_{\lambda}(k)\boxtimes V_{\lambda}(m-k),
    \] where $\boxtimes$ denotes the external tensor product of representations and $\ell(\lambda)$ is the number of rows in the Young diagram of $\lambda$.
\end{prop}

The dimension of each $V_{\lambda}(k)\boxtimes V_{\lambda}(m-k)$ is $|{\sf SSYT}(\lambda,k)\times {\sf SSYT}(\lambda,m-k)|$.
In~\cite[Chapter 11.8]{BCRV}, the authors study the family of ideals $I_{\lambda} = \langle V_{\lambda}(k)\boxtimes V_{\lambda}(m-k)\rangle$. 
Our $I_{m,k}$ happens to be equal to $I_\lambda$ for $\lambda = (3)$.
This is a special case of a claim made in~\cite[Example 11.8.5]{BCRV}.\footnote{In \cite{BCRV} the conventions for partitions are transposed from ours, so there our $I_{(3)}$ is denoted $I_{(1, 1, 1)}$.} 
We provide a proof for completeness in Proposition~\ref{prop:decomposition}. We then apply this fact to prove that the generators \eqref{eq:gens} of $I_{m,k}$ form a Gr{\"o}bner basis in Proposition~\ref{prop:grobner}. 

Our proofs rely on the \emph{Robinson--Schensted--Knuth correspondence}, a classical bijection between nonnegative integer matrices and pairs of semistandard Young tableaux of the same shape. 
We do not present the bijection in detail here, but clarify our convention choices.
Given a nonnegative integer matrix $M\in\Mat_{a, b}(\Z_{\geq 0})$, its \emph{column word} is the word on the alphabet $[b]$ formed by recording $M_{ij}$ copies of the column index $j$ for each entry $M_{ij}$ of $M$. 
Entries are read in English reading order, i.e., in the lexicographic order from the introduction. 
The \emph{row word} of $M$ is the column word of its transpose. 
Given a word $\alpha$, let ${\sf tab}(\alpha)$ be the tableau obtained by starting with the empty tableau and successively row-inserting each letter of $\alpha$ (read from left to right) according to the standard algorithm in \cite[Section 1.1]{Fulton}. 
We then define
\[\mathrm{RSK}(M) \coloneq ({\sf tab}({\sf row}(M)), {\sf tab}({\sf col}(M))).\footnote{This differs slightly from the RSK of \cite[Chapter 4]{Fulton}, where the pair of tableaux in the output are swapped.}\]

\begin{definition}
    Let $M\in{\sf Mat}_{a, b}(\Z_{\geq 0})$ be an $a\times b$ nonnegative integer matrix. A \emph{weak diagonal} of $M$ is a sequence of matrix positions $D = ((i_1,j_1),\dots,(i_r, j_r))$, such that $i_1\leq\dots\leq i_r$, $j_1\leq\dots\leq j_r$, and the number of occurrences of each position $(i, j)$ in $D$ does not exceed $M_{i, j}$. The \emph{width} of $M$, ${\sf width}(M)$, is the maximum length of a weak diagonal of $M$.
\end{definition}

\begin{lemma}\label{lemma:diag}
    For any $M\in{\sf Mat}_{a, b}(\Z_{\geq 0})$, ${\sf width}(M)$ is the length of the longest weakly increasing subsequence in ${\sf col}(M)$. 
\end{lemma}
\begin{proof}
    By the definition of ${\sf col}(M)$, each entry $M_{ij}$ of $M$ contributes $M_{ij}$ ``$j$'' letters to ${\sf col}(M)$. 
    Let $M_{ij}$ and $M_{i'j'}$ be distinct nonzero entries of $M$, assuming without loss of generality that $j\leq j'$. Then the associated $j$ and $j'$ letters in ${\sf col}(M)$ form a weakly increasing subword if and only if matrix position $(i, j)$ occurs before $(i',j')$ in the English reading order on matrix positions. 
    Since $j\leq j'$, this occurs if and only if $i\leq i'$, i.e., if and only if $(i, j)$ and $(i', j')$ form a weak diagonal of $M$. 
    The result follows.
\end{proof}

\begin{definition}
    The \emph{conjugate} of a partition $\lambda$, denoted $\lambda'$, is the partition whose Young diagram is obtained by transposing the rows and columns of the Young diagram of $\lambda$.
\end{definition}

\begin{theorem}[RSK]\label{thm:RSK}
    For all nonnegative integers $a, b$ and $c$, the map $\mathrm{RSK}$ defined above gives a bijection
    \[\{M\in{\sf Mat}_{a, b}(\Z_{\geq 0}) \,:\, {\sf width}(M)= c\}\xleftrightarrow{\mathrm{RSK}}\bigsqcup_{\ell(\lambda')=c}\left({\sf SSYT}(\lambda, a)\times {\sf SSYT}(\lambda, b)\right).\]
\end{theorem}
\begin{proof}
    The fact that ${\rm RSK}$ defines a bijection
    
    \[{\sf Mat}_{a, b}(\Z_{\geq 0})\xleftrightarrow{\mathrm{RSK}} \bigsqcup_{\lambda}\left({\sf SSYT}(\lambda, a)\times {\sf SSYT}(\lambda, b)\right)\]
    is standard, see, e.g. \cite[Theorem~7.11.5]{ECII} or \cite[R-S-K~Theorem]{Fulton}. 
    Schensted \cite[Theorem 1 and Section II]{Schensted} proved that for any word $\alpha$, the number of columns in ${\sf tab}(\alpha)$ is equal to the length of the longest weakly increasing subword of $\alpha$.
    Lemma~\ref{lemma:diag} states that ${\sf width}(M)$ is the length of the longest increasing subword of ${\sf col}(M)$, completing the proof. 
\end{proof}

\begin{definition}\label{def:perm}
    A \emph{$k\times k$ generalized permanent} of a matrix $X = (X_{ij})$ is the permanent of some $k\times k$ matrix
    \[\begin{bmatrix}X_{i_1j_1} & X_{i_1j_2} & \dots & X_{i_1j_k}\\
    X_{i_2j_1} & X_{i_2j_2} & \dots & X_{i_2j_k}\\
    \vdots & \vdots & \ddots & \vdots\\
    X_{i_kj_1} & X_{i_kj_2} & \dots & X_{i_kj_k}\end{bmatrix},\]
    where $i_1\leq i_2\leq\dots\leq i_k$ and $j_1\leq j_2\leq\dots\leq j_k$. 
    We observe that the lead term of a generalized permanent under our lexicographic order is the diagonal term $X_{i_1,j_1}\dots X_{i_k,j_k}$.
\end{definition}

In what follows, we will index monomials in $\C[X]$ by their exponent vectors, written as $k\times (m-k)$ nonnegative integer matrices, and we use $\dim_\C(I)_d$ to refer to the dimension of the degree-$d$ component of an ideal $I$ as a $
\C$-vector space. 
For a polynomial $f\in \C[X]$, let ${\sf LT}(f)$ denote the lead term of $f$ with respect to the lexicographic term order from the introduction. 
The \emph{initial ideal} of an ideal $I\subseteq\C[X]$ is then
\[\init(I) = \langle{\sf LT}(f): f\in I\rangle\subseteq\C[X].\]
Let $F_{m,k} \subseteq I_{m,k}$ be the finite set of generators \eqref{eq:gens}.
We consider the initial ideal
\[{\sf init}(F_{m,k}) = \langle {\sf LT}(f_{p, q, r}^{a, b, c}):f_{p, q, r}^{a, b, c}\text{ as in }\eqref{eq:gens}\rangle = \langle X_{p, a}X_{q, b}X_{r, c}:p\leq q\leq r, a\leq b\leq c\rangle\subseteq\C[X].\]
Since each $f_{p, q, r}^{a, b, c}$ lies in $I_{m, k}$, we know that ${\sf init}(F_{m,k} )\subseteq\init(I_{m, k})$, with equality holding if and only if the $f_{p, q, r}^{a, b, c}$ form a Gr\"obner basis for $I_{m, k}$ (see \cite[Chapter~2]{CoxLittleOSheaSweedler} for background on Gr\"obner~bases). 

\begin{lemma}\label{lemma:permdiag}
    Let $I_{m, k}$ and $F_{m,k}$ be as above. Then the following properties hold:
    \begin{enumerate}
        \item[(I)] As a vector space, ${\sf init}(F_{m,k})$ has a basis consisting of all monomials of the form $X^M$ with ${\sf width}(M)\geq 3$. 
        \item[(II)] For all $d\geq 0$ we have
       \[ \dim_{\C}({\sf init}_{<}(F_{m,k})_d) = \sum_{\lambda:\ell(\lambda') \geq 3,|\lambda|=d}|{\sf SSYT}(\lambda,k)\times {\sf SSYT}(\lambda,m-k)|.\]
        \item[(III)] For all $d\geq 0$, $\dim_\C({\sf init}(F_{m,k})_d)\leq \dim_\C((I_{m, k})_d)$, with equality holding for $d = 3$. 
    \end{enumerate}
\end{lemma}
\begin{proof}
    (I): Note that ${\sf init}(F_{m, k})$ has a vector space basis consisting of all monomials $X^M$ divisible by some term of the form $X_{p, a}X_{q, b}X_{r, c}$, where $p\leq q\leq r$ and $a\leq b\leq c$. Then $((p, a), (q, b), (r, c))$ is a weak diagonal of $M$ of length $3$, so ${\sf width}(M)\geq 3$. 
    Conversely, if $X^M$ satisfies ${\sf width}(M)\geq 3$, then there exists some weak diagonal $D = ((p, a),(q, b),(r, c))$ in $M$ of length $3$. 
    Then $X^M$ is divisible by $X_{p, a}X_{q, b}X_{r, c}$, so $X^M$ lies in ${\sf init}(F_{m, k})$. 
    Thus ${\sf init}(F_{m, k})$ has the claimed basis. 

    (II): This is immediate from applying Theorem~\ref{thm:RSK} to the basis from part (I). 

    (III): It is a standard fact (see, e.g., \cite[Proposition 5.3.4]{CoxLittleOSheaSweedler}) that for any ideal $I$ and degree $d$, $\dim_\C(\init(I)_d) = \dim_\C(I_d)$. 
    The desired inequality follows, since ${\sf init}(F_{m, k})\subseteq \init(I_{m, k})$ by its construction. 
    When $d = 3$, we know that $\dim_\C((I_{m, k})_3)$ is bounded above by the number of ideal generators $f_{p, q, r}^{a, b, c}$ in \eqref{eq:gens}. 
    Since each $f_{p, q, r}^{a, b, c}$ has a distinct lead term in ${\sf init}(F_{m, k})$, it follows that $\dim_\C((I_{m, k})_3)\leq \dim_\C({\sf init}(F_{m, k})_3)$ also, proving equality.
\end{proof}

\begin{prop}\label{prop:decomposition}
    As ideals and as $\mathrm{GL}_k(\C)\times \mathrm{GL}_{m-k}(\C)$-representations, 
    \[
    I_{m,k} = \langle V_{(3)}(k)\boxtimes V_{(3)}(m-k)\rangle.
    \]
\end{prop}
\begin{proof}
    First, we show that $I_{m,k}$ is $\mathrm{GL}_k(\C)\times \mathrm{GL}_{m-k}(\C)$-stable. It suffices to show that, for any $(A,B)\in \mathrm{GL}_k(\C)\times \mathrm{GL}_{m-k}(\C)$ and any $f_{p,q,r}^{a,b,c}$, $(A,B)\cdot f_{p,q,r}^{a,b,c}\in I_{m,k}$. This is because every element $g\in I_{m,k}$ is of the form
    \[
    g(X) = \sum_{p,q,r,a,b,c} h_{p,q,r}^{a,b,c}(X)f_{p,q,r}^{a,b,c}(X) \text{ for some polynomials } h_{p,q,r}^{a,b,c}(X)\in \C[X],
    \] so if $(A,B)\cdot f_{p,q,r}^{a,b,c}(X)\in I_{m,k}$ for each $f_{p,q,r}^{a,b,c}(X)$, then
    \[
    (A,B)\cdot g(X) = \sum_{p,q,r,a,b,c}h_{p,q,r}^{a,b,c}(AXB^{T})f_{p,q,r}^{a,b,c}(AXB^{T})\in I_{m,k}.
    \] 
    The ideal generators (\ref{eq:gens}) are precisely the $3\times 3$ generalized permanents of $X$ as given in Definition~\ref{def:perm}. 
    Since the permanent is a multilinear function, $(A,B)\cdot f_{p,q,r}^{a,b,c}(X) = f_{p,q,r}^{a,b,c}(AXB^{T})$ may be written as a linear combination of generalized permanents of $X$. Thus, $(A,B)\cdot f_{p,q,r}^{a,b,c}(X)\in I_{m,k}$, proving that $I_{m,k}$ is $\mathrm{GL}_k(\C)\times \mathrm{GL}_{m-k}(\C)$-stable.

    Since $I_{m,k}$ contains the monomial $X_{1,1}^3$ and is $\mathrm{GL}_k(\C)\times \mathrm{GL}_{m-k}(\C)$-stable, by~\cite[Theorem 3.6.8(b)]{BCRV}, $I_{m,k}\supseteq \langle V_{(3)}(k)\boxtimes V_{(3)}(m-k)\rangle$. To show that $I_{m,k} = \langle V_{(3)}(k)\boxtimes V_{(3)}(m-k)\rangle$, it now suffices to show that the degree $3$ component $(I_{m,k})_3$ of $I_{m,k}$ has dimension equal to 
    \[
    \dim_\C (V_{(3)}(k)\boxtimes V_{(3)}(m-k)) = |{\sf SSYT}((3),k)\times {\sf SSYT}((3),m-k)|.
    \] 
    This follows immediately by parts (II) and (III) of Lemma~\ref{lemma:permdiag}.
\end{proof}

\begin{prop}\label{prop:grobner}
    The ideal generators \eqref{eq:gens} of $I_{m,k} \subseteq \C[X]$ form a Gröbner basis with respect to the lexicographic monomial order.
\end{prop}
\begin{proof}
    By construction 
    ${\sf init}(F_{m, k})\subseteq\init(I_{m, k})$,
    and the proposition statement is equivalent to the claim that equality holds. 
    We prove this by demonstrating that for all $d\geq 0$, the degree-$d$ components of ${\sf init}(F_{m, k})$ and $\init(I_{m, k})$ have the same vector space dimensions. 
    By \cite[Proposition 5.3.4]{CoxLittleOSheaSweedler} and Proposition~\ref{prop:decomposition}, 
    \[\dim_\C((\init(I_{m, k}))_d) = \dim_\C((I_{m, k})_d) = \dim_\C\langle V_{(3)}(k)\boxtimes V_{(3)}(m-k)\rangle_d.\]
    By~\cite[Corollary 11.8.3]{BCRV}, we have
    \[
    \dim_\C\langle V_{(3)}(k)\boxtimes V_{(3)}(m-k)\rangle_d = \sum_{\lambda:\ell(\lambda') \geq 3,|\lambda|=d}|{\sf SSYT}(\lambda,k)\times {\sf SSYT}(\lambda,m-k)|.
    \] 
    Lemma~\ref{lemma:permdiag}(II) shows that this is also the dimension of $({\sf init}(F_{m, k}))_d$, as claimed.
\end{proof}
\begin{remark}
    Our proof of Proposition~\ref{prop:grobner} is analogous to that used by Sturmfels \cite{Sturmfels} to prove that the set of all $k\times k$ minors of a generic matrix form a Gr\"obner basis for the \emph{classical determinantal ideal} they generate. 
    The argument extends, by replacing all instances of $3$ with $p$, to show that the $p\times p$ generalized permanents of $X$ form a Gr{\"o}bner basis. 
    The connection between ideals generated by determinants and generalized permanents has been observed previously and is sometimes known as ``transpose duality''; see, e.g., \cite[Section~2.3]{Permanents}. We believe that Proposition~\ref{prop:grobner} has not appeared in the literature before, although all elements of our proof exist in the literature.
\end{remark}

\begin{remark}
    Our proof of Proposition~\ref{prop:grobner} works by using knowledge of the $\mathrm{GL}_k(\C)\times \mathrm{GL}_{m-k}(\C)$-equivariant Hilbert series of $S_{m, k}$ and RSK to prove that the defining generators \eqref{eq:gens} form a Gr\"obner basis. 
    Conversely, the equivariant Hilbert series of $S_{m, k}$ can be derived from RSK and knowledge of this Gr\"obner basis, using the method of \emph{Gr\"obner crystal structures} developed by Price, Stelzer, and Yong in \cite{AAA2}. 
    In the language of that paper, one can show that the ideals generated by $p\times p$ generalized permanents of a generic matrix $X$ are \emph{bicrystalline} with respect to diagonal term orders, which is analogous to the result \cite[Theorem~6.3]{AAA2} that ideals generated by $k\times k$ determinants are bicrystalline with respect to antidiagonal term orders. 
    See \cite[Example~2.10]{AAA2} for an example of this phenomenon.
\end{remark}

\begin{corollary}\label{cor:stdmono}
    The standard monomial basis of $S_{m, k}$ is indexed by the set
    \[\{M\in{\sf Mat}_{k, m-k}(\Z_{\geq 0}) \,:\, {\sf width}(M)\leq 2\}.\]
\end{corollary}
\begin{proof}
    Proposition~\ref{prop:grobner} shows that the defining generators (\ref{eq:gens}) of $I_{m, k}$ form a Gr\"obner basis. By definition, the standard monomial basis of $S_{m,k}$ then consists of those monomials lying outside of $\init(I_{m, k}) = {\sf init}(F_{m, k})$ \cite[Proposition 5.3.4]{CoxLittleOSheaSweedler}. 
    These are precisely the monomials $X^M$ such that ${\sf width}(M)\leq 2$ by Lemma~\ref{lemma:permdiag}(I). 
\end{proof}

\begin{example}
    Let $m = 4$ and $k = 2$. The ring $S_{4,2}$ has $20$ standard monomials. Since the ideal generators are of degree three, we can see that all $15$ monomials of degree at most two in $\C[X]$ are standard monomials for $S_{4, 2}$. Of these, $10 = \binom{4 +1}{2}$ are monomials of degree two, while the others are the four variables $X_{i,j}$ and the constant $1$. The $10$ monomials of degree two correspond to the following matrices in ${\sf Mat}_{2, 2}(\Z_{\geq 0})$:
    $$
    \begin{pmatrix}
        2 & 0 \\
        0 & 0
    \end{pmatrix}, \, \begin{pmatrix}
        0 & 2\\
        0 & 0
    \end{pmatrix}, \, \begin{pmatrix}
        0 & 0\\
        2 & 0
    \end{pmatrix}, \, \begin{pmatrix}
        0 & 0 \\
        0 & 2
    \end{pmatrix}, \, \begin{pmatrix}
        1 & 1 \\
        0 & 0
    \end{pmatrix}, \, \begin{pmatrix}
        1 & 0\\
        1 & 0
    \end{pmatrix}, \, \begin{pmatrix}
        1 & 0\\
        0 & 1
    \end{pmatrix}, \, \begin{pmatrix}
        0 & 1 \\
        1 & 0
    \end{pmatrix}, \, \begin{pmatrix}
        0 & 1\\
        0 & 1
    \end{pmatrix}, \, \begin{pmatrix}
        0 & 0\\
        1 & 1
    \end{pmatrix}.
    $$
    There are four standard monomials of degree three: $X_{1,1}X_{1,2}X_{2,1}$, $X_{1,2}^2X_{2,1}$, $X_{1,2}X_{2,1}^2$ and $X_{1,2}X_{2,1}X_{2,2}$. They correspond respectively to the matrices 
    $$
    \begin{pmatrix}
        1 & 1\\
        1 & 0 
    \end{pmatrix}, \,\,\, \begin{pmatrix}
        0 & 2\\
        1 & 0
    \end{pmatrix}, \,\,\, \begin{pmatrix}
        0 & 1\\
        2 & 0
    \end{pmatrix}, \,\,\,
     \begin{pmatrix}
        0 & 1\\
        1 & 1
    \end{pmatrix}.
    $$
    The only standard monomial of degree four is $X_{1,2}^2 X_{2,1}^2$, which corresponds to
    $$
    \begin{pmatrix}
        0 & 2\\
        2 & 0
    \end{pmatrix}.
    $$
\end{example}

In order to prove Theorem~\ref{thm:degofex}, it remains to enumerate our standard monomial basis for $S_{m, k}$. We carry this out combinatorially in the next section.

\section{RSK and plane partitions}\label{sec:bij}

To enumerate the basis for $S_{m, k}$ from Corollary~\ref{cor:stdmono}, we map the standard monomials to a well-studied set of \emph{plane partitions} by applying a version of RSK different to that stated in Theorem~\ref{thm:RSK}. We follow Hopkins' exposition in \cite{Hopkins} to state this alternative version of RSK.

\begin{definition}
    An $a\times b\times c$ \emph{plane partition} is an $a\times b$ nonnegative integer matrix whose entries are bounded above by $c$ and are weakly decreasing along rows and columns. The set of all such objects is denoted by $\mathcal{B}(a, b, c)$.
\end{definition}
\begin{example}
    The matrix $\left[\begin{smallmatrix}
        5 & 4 & 3 & 1\\
        3 & 2 & 2 & 1\\
        3 & 2 & 1 & 0\\
        2 & 1 & 1 & 0
    \end{smallmatrix}\right]$ is a $4\times 4\times c$ plane partition for any $c\geq 5$.
\end{example}

\begin{theorem}[{\cite[Theorem~5]{Pak}}]\label{thm:hopkinsRSK}
    For all $a, b, c\in\Z_{\geq 0}$, there is a bijection 
    \[\bigsqcup_{\ell(\lambda')\leq c}\left({\sf SSYT}(\lambda, a)\times {\sf SSYT}(\lambda, b)\right)\leftrightarrow\mathcal{B}(a, b, c).\] 
\end{theorem}
\begin{proof}[Proof Sketch]
    We use the standard bijection described, e.g., by Pak in \cite[Theorem~5]{Pak}, with one difference of convention: his bijection sends pairs $(P, Q)\in{\sf SSYT}(\lambda, a)\times{\sf SSYT}(\lambda, b)$ to $a\times b$ matrices that are weakly \emph{increasing} along rows and columns rather than decreasing. 
    To obtain an $a\times b$ plane partition, we rotate the output of Pak's bijection by $180^\circ$. 

    Explicitly, our bijection is as follows: given a tableau $P\in{\sf SSYT}(\lambda, n)$, for each $1\leq j\leq i\leq n$ let $P_j^i$ denote the number of entries $\leq i$ in row $j$ of $P$.\footnote{Note that if $j>i$, there are no entries $\leq i$ in row $j$ of $P$. The triangular array of numbers $P_j^i$ is called the \emph{Gelfand--Tsetlin pattern} associated to $P$.} 
    Now associate to each pair $(P, Q)\in {\sf SSYT}(\lambda, a)\times {\sf SSYT}(\lambda, b)$ the $a\times b$ plane partition $B = [B_{ij}]$ such that
    \[B_{ij} = \begin{cases}
        P^{a-i+j}_j & \text{if } i\geq j,\\
        Q^{b-j+i}_i & \text{if } i\leq j.
    \end{cases}\]
    The inverse map takes an $a\times b$ plane partition $B = [B_{ij}]$ as input and outputs a tableau-pair $(P, Q)$, where $P$ is the unique semistandard tableau with $B_{a+j-i, j}-B_{a+1+j-i,j}$ ``$i$'' entries in row $j$ for all $j\leq i\leq a$, and $Q$ is the unique semistandard tableau with $B_{j, b+j-i}-B_{j, b+1+j-i}$ ``$i$'' entries in row $j$ for all $j\leq i\leq b$.
    Note that under these maps, the number of columns in the partition shape $\lambda$ is equal to the maximum entry $B_{11}$ in the corresponding plane partition, as claimed.
\end{proof}

The following example illustrates the bijection between pairs of semistandard Young tableaux and plane partitions described in the proof of Theorem \ref{thm:hopkinsRSK}.

\ytableausetup{smalltableaux,centertableaux}
\begin{example}
    Consider the pair of tableaux $(P,Q) = \left(\begin{ytableau}
        1 & 1 & 2 & 4 & 4\\
        2 & 3\\
        3
    \end{ytableau}, \begin{ytableau}
        1 & 2 & 2 & 3 & 4\\
        2 & 3\\
        4
    \end{ytableau}\right),$ both with entries in $\{1,2,3,4\}$. 
    The corresponding plane partition is 
    \[\begin{bmatrix}
    P_1^4 = Q_1^4 & Q_1^3 & Q_1^2 & Q_1^1\\
    P_1^3 & P_2^4 = Q_2^4 & Q_2^3 & Q_2^2\\
    P_1^2 & P_2^3 & P_3^4 = Q_3^4 & Q_3^3\\
    P_1^1 & P_2^2 & P_3^3 & P_4^4 = Q_4^4
    \end{bmatrix}
    =
    \begin{bmatrix}
        5 & 4 & 3 & 1\\
        3 & 2 & 2 & 1\\
        3 & 2 & 1 & 0\\
        2 & 1 & 1 & 0
    \end{bmatrix}.\] 
    The largest entry of this plane partition is $5$ precisely because $P$ and $Q$ have $5$ columns.
\end{example}

\begin{remark}
    There is a direct bijection between nonnegative integer matrices and plane partitions which does not require first passing through semistandard Young tableaux. For details, we direct the reader to Hopkins' exposition in~\cite{Hopkins}.
\end{remark}

The following enumeration for $\mathcal{B}(a, b, c)$ is classical, deriving from MacMahon's generating series for plane partitions \cite{MacMahon}. See, e.g., \cite[Equation (7.109)]{ECII} for a modern statement:
\begin{equation}\label{eqn:PP}
    |\mathcal{B}(a, b, c)| = \prod_{i=1}^c\frac{\binom{a+b+i-1}{a+i-1}\binom{a+b+i-1}{b+i-1}}{\binom{a+b+i-1}{a}\binom{b+i-1}{b}}.
\end{equation}
This formula suffices to prove Theorem~\ref{thm:degofex}.

\begin{proof}[Proof of Theorem~\ref{thm:degofex}]
    We wish to show that the excitation ring $S_{m, k}$ has vector space dimension $N(m+1, k+1) = \frac{1}{m+1}\binom{m+1}{k+1}\binom{m+1}{k}$. By Corollary~\ref{cor:stdmono}, it suffices to show that
    \begin{equation}\label{eqn:enum}
        |\{M\in{\sf Mat}_{k, m-k}(\Z_{\geq 0}) \,:\, {\sf width}(M)\leq 2\}| = N(m+1, k+1).
    \end{equation}
    Chaining together the bijections from Theorems~\ref{thm:RSK} and ~\ref{thm:hopkinsRSK} shows that 
    \[|\{M\in{\sf Mat}_{k, m-k}(\Z_{\geq 0}) \,:\, {\sf width}(M)\leq 2\}| = |\mathcal{B}(k,m-k,2)|.\]
    The enumeration in \eqref{eqn:PP} then proves that
$|\mathcal{B}(k,m-k,2)| = N(m+1,k+1)$:
\begin{align*}
    |\mathcal{B}(k, m-k, 2)| &= \frac{\binom{k+(m-k)+1-1}{k+1-1}\binom{k+(m-k)+1-1}{(m-k)+1-1}}{\binom{k+(m-k)+1-1}{k}\binom{(m-k)+1-1}{m-k}}\cdot \frac{\binom{k+(m-k)+2-1}{k+2-1}\binom{k+(m-k)+2-1}{(m-k)+2-1}}{\binom{k+(m-k)+2-1}{k}\binom{(m-k)+2-1}{m-k}}\\
    &= \frac{\binom{m}{k}\binom{m}{m-k}}{\binom{m}{k}\binom{m-k}{m-k}}\cdot\frac{\binom{m+1}{k+1}\binom{m+1}{m-k+1}}{\binom{m+1}{k}\binom{m-k+1}{m-k}}\\
    &= \frac{\binom{m}{k}}{1}\cdot \frac{\binom{m+1}{k+1}}{m-k+1}\\
    &= \frac{1}{m+1}\binom{m+1}{k}\binom{m+1}{k+1} = N(m+1,k+1).\qedhere
\end{align*}
\end{proof}

The remainder of this section presents another bijection, combinatorially explaining the appearance of the Narayana numbers $N(m+1, k+1)$ in Theorem~\ref{thm:degofex}. 
We recall the standard combinatorial interpretation for these numbers.

\begin{definition}
    A \emph{Dyck word} of length $2n$ is a string $w$ of $n$ ``$u$'' symbols and $n$ ``$d$'' symbols that is \emph{ballot}, meaning that every initial substring of $w$ contains at least as many $u$'s as $d$'s. A \emph{valley} in $w$ is a $du$-substring. Let $\mathcal{D}(n, r)$ denote the set of Dyck words of length $2n$ with exactly $r-1$ valleys. 
\end{definition}

The term ``valley'' originates from the visual interpretation of Dyck words as walks in the first quadrant from $(0, 0)$ to $(2n, 0)$, consisting of ``up'' moves $(i, j)\to(i+1, j+1)$ and ``down'' moves $(i, j)\to(i+1, j-1)$. 
See Figure~\ref{fig:dyckpathnoblue} for an example.

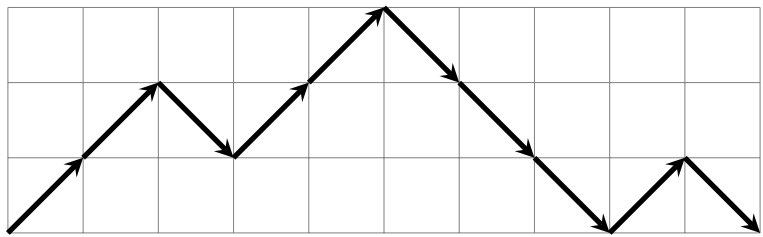
\begin{figure}[h!]
    \centering
    \begin{tikzpicture}
        \dyckpath{0,0}{10}{3}{0,0,1,0,0,1,1,1,0,1}
    \end{tikzpicture}
    \caption{The element $w = uuduudddud$ of $\mathcal{D}(5,3)$.}
    \label{fig:dyckpathnoblue}
\end{figure}

The Narayana number $N(n, r)$ classically arises as the cardinality of $\mathcal{D}(n, r)$, see e.g.\ \cite[Section 2.4.2]{Petersen}.
They define a refinement of the Catalan numbers $C_n$, which count the number of Dyck words of length $2n$ with any number of valleys. 
By definition,
\[\sum_{r = 1}^n N(n,r) = \frac{1}{n + 1}\binom{2n}{n} = C_{n}.\]

In order to give an explicit bijection between $\mathcal{D}(m+1, k+1)$ (the set of Dyck paths of length $2m+2$ with precisely $k$ valleys) and $\mathcal{B}(k, m-k, 2)$, we introduce some additional terminology. 
A letter in a Dyck word $w$ is \emph{central} if it is not the first or last letter of $w$, and it is not part of a valley of $w$. 
Let ${\sf up}_i(w)$ and ${\sf down}_i(w)$ respectively denote the number of central $u$'s and $d$'s occurring before the $i$th valley of $w$. 
If $i$ exceeds the number of valleys in $w$, ${\sf up}_i(w)$ and ${\sf down}_i(w)$ are defined to count the total number of central $u$'s and $d$'s in $w$.

\begin{theorem}\label{thm:main}
    The map $\mathcal{D}(m+1, k+1)\to\mathcal{B}(k, m-k,2)$ sending $w\mapsto B^w$, where
    \begin{equation}\label{eqn:mapdef}
        B^w_{ij} = \begin{cases}
        2 & \text{if } {\sf down}_{k+1-i}(w)\geq j,\\
        1 & \text{if } {\sf down}_{k+1-i}(w)< j\text{ and } {\sf up}_{k+1-i}(w)\geq j,\\
        0 & \text{if } {\sf up}_{k+1-i}(w)< j,
    \end{cases}
    \end{equation}
    is a well-defined bijection.
\end{theorem}
\begin{proof}
    We begin by observing that a Dyck word $w\in\mathcal{D}(m+1,k+1)$ is uniquely determined by the sequence $({\sf up}_i(w), {\sf down}_i(w))_{i=1}^k$. By construction, a Dyck word $w\in\mathcal{D}(m+1,k+1)$ is determined by its sequence of central letters, which is of the following form:
    \begin{equation}\label{eqn:dyckform}
        \overbrace{u\dots u}^{{\sf up}_1(w)}\underbrace{d\dots d}_{{\sf down}_1(w)}\hspace{-0.5cm}\overbrace{u\dots u}^{{\sf up}_2(w)-{\sf up}_1(w)}\hspace{-1.2cm}\underbrace{d\dots d}_{{\sf down}_2(w)-{\sf down}_1(w)}\hspace{-0.5cm}\cdots\hspace{-0.5cm}\overbrace{u\dots u}^{{\sf up}_{k+1}(w)-{\sf up}_k(w)}\hspace{-1.7cm}\underbrace{d\dots d}_{{\sf down}_{k+1}(w)-{\sf down}_k(w)}\hspace{-1cm}.
    \end{equation}
    We see immediately that $({\sf up}_i)_{i=1}^{k+1}$ and $({\sf down}_i)_{i=1}^{k+1}$ are weakly increasing sequences, and the ballot condition on Dyck words implies that ${\sf up}_i(w)\geq{\sf down}_i(w)$ for all $i$. 
    Moreover, since $w$ has $2m+2$ total letters and $k$ valleys, the number of central letters in $w$ is $(2m+2)-2k-2 = 2(m-k)$. 
    Thus ${\sf up}_{k+1}(w) = {\sf down}_{k+1}(w) = m-k$. 
    Conversely, given any pair of weakly increasing sequences $(a_i)_{i=1}^k$ and $(b_i)_{i=1}^k$ satisfying $a_i\geq b_i$ for all $i$ and $m-k\geq a_k$, we may set $a_{k+1} = b_{k+1} = m-k$ and construct a word of the form \eqref{eqn:dyckform}. 
    Inserting an initial $u$, final $d$, and valleys produces a Dyck word $w\in\mathcal{D}(m+1,k+1)$ such that $a_i = {\sf up}_i(w)$ and $b_i = {\sf down}_i(w)$.

    We now show that $B^w$ is in fact a plane partition. Fix any $1\leq i\leq k$, and let $1\leq j<j' \leq m-k$. If ${\sf up}_{k+1-i}(w) < j$, then certainly ${\sf up}_{k+1-i}(w) < j'$. So, if $B_{ij}^w = 0$ for some $j$, $B_{ij'}^w = 0$ for all $j' > j$. Similarly, if ${\sf down}_{k+1-i}(w) < j$, then ${\sf down}_{k+1-i}(w) < j'$. So, if $B_{ij}^w = 1$ for some $j$, $B_{ij'}^w \leq 1$ for all $j' > j$. This shows that $B^w$ weakly decreases along rows.

    Fix any $1\leq j\leq m-k$, and let $1\leq i < i' \leq k$. If ${\sf up}_{k+1-i}(w) < j$, then, since $({\sf up}_i)_{i=1}^k$ is a weakly increasing sequence and $k+1-i > k+1-i'$, ${\sf up}_{k+1-i'}(w) < j$. So, if $B_{ij}^w = 0$ for some $i$, $B_{i'j} = 0$ for all $i' > i$. Similarly, if ${\sf down}_{k+1-i}(w) < j$, then ${\sf down}_{k+1-i'}(w) < j$. So, if $B_{ij}^w = 1$ for some $i$, then $B_{i'j}^w\leq 1$ for all $i' > i$. This shows that $B^w$ weakly decreases along columns. Thus, $B^w$ is in fact a plane partition.

    To complete the proof of the theorem, it suffices to construct an inverse map by extracting weakly increasing sequences $(a_i^B)_{i=1}^k$ and $(b_i^B)_{i=1}^k$ from each $B\in\mathcal{B}(k,m-k,2)$ such that $a^B_i\geq b^B_i$ for all $i$ and $m-k\geq a_k^B$. 
    For each $i\in[k]$, let $c_i$ be the largest integer such that $B_{i, c_i} = 2$ (unless $B_{i, 1}<2$, in which case set $c_i = 0$) and let $c'_i$ be the smallest integer such that $B_{i, c'_i+1} = 0$ (or, if $B_{i, m-k}>0$, set $c'_i = m-k$). 
    Then define $a^B_i = c'_{k+1-i}$ and $b^B_i = c_{k+1-i}$. By construction, we have that 
    \begin{equation}\label{eqn:planepartitions}
        B_{ij} = \begin{cases}
        2 & \text{if } b^B_{k+1-i}\geq j,\\
        1 & \text{if } b^B_{k+1-i}< j\text{ and } a^B_{k+1-i}\geq j,\\
        0 & \text{if } a^B_{k+1-i}< j.
    \end{cases}
    \end{equation}
    
    These sequences satisfy the desired conditions. Since $B$ weakly decreases along columns, both $(a_i^B)_{i=1}^k$ and $(b_i^B)_{i=1}^k$ are weakly increasing sequences. Since $B$ is weakly decreasing along rows, $c'_i \geq c_i$ for all $i$. So, $a_i^B\geq b_i^B$ for all $i$. Also, since $B$ has $m-k$ columns, $m-k\geq a_k^B$. Denote by $w_B$ the word formed using $(a_i^B)_{i=1}^k$ and $(b_{i}^B)_{i=1}^k$; we have shown that $w_B\in \mathcal{D}(m+1,k+1)$.

    Finally, we claim that the map $B \mapsto w_B$ is the inverse of $w\mapsto B^w$. By Equations~\eqref{eqn:planepartitions} and \eqref{eqn:mapdef}, $w_{B^w} = w$. We now show that $B^{w_B} = B$. Form $(a_i^{B})_{i=1}^k$ and $(b^{B}_i)_{i=1}^k$ as above, and set $a^{B}_{k+1} = b_{k+1}^{B} = m-k$. Form the sequences $({\sf up}_i(w_B))_{i=1}^{k+1}$ and $({\sf down}_i(w_B))_{i=1}^{k+1}$. Now,
    \[B^{w_B}_{ij} = \begin{cases}
        2 & \text{if } {\sf down}_{k+1-i}(w_B)\geq j,\\
        1 & \text{if } {\sf down}_{k+1-i}(w_B)< j\text{ and } {\sf up}_{k+1-i}(w_B)\geq j,\\
        0 & \text{if } {\sf up}_{k+1-i}(w_B)< j.
    \end{cases}\] By the construction of $(a^B_i)_{i=1}^{k+1}$ and $(b^B_i)_{i=1}^{k+1}$, 
    \[
    {\sf up}_i(w_B) = a_i^B \text{ and } {\sf down}_i(w_B) = b_i^B.
    \] Thus, $B^{w_B} = B$. This completes the proof.
\end{proof}

The following example demonstrates the bijection of Theorem \ref{thm:main}. 

\begin{example}
    Let $m = 4$ and $k = 2$, and consider the Dyck word $w = uuduudddud$ illustrated in Figure~\ref{fig:dyckpath} (central paths are highlighted in blue). 
    We can compute that 
    \[{\sf up}_1(w) = 1, \; {\sf up}_2(w) = 2, \; {\sf down_1}(w) = 0, \; {\sf down}_2(w) = 2.\] The corresponding plane partition $B^w$ is $\left[\begin{smallmatrix}
        2 & 2\\
        1 & 0
    \end{smallmatrix}\right]$. For instance, $B^w_{21} = 1$ since ${\sf down}_{3-2}(w) = {\sf down}_1(w) = 0 < 1$ and ${\sf up}_{3-2}(w) = {\sf up}_1(w) = 1\geq 1$.

    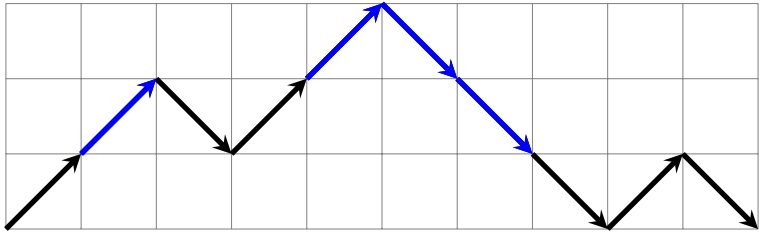
\begin{figure}[h!]
    \centering
    \begin{tikzpicture}
        \dyckpath{0,0}{10}{3}{0,0,1,0,0,1,1,1,0,1}
        \draw[line width=2pt,-stealth,blue] (1,1) -- (2,2);
        \draw[line width=2pt,-stealth,blue] (4,2) -- (5,3);
        \draw[line width=2pt,-stealth,blue] (5,3) -- (6,2);
        \draw[line width=2pt,-stealth,blue] (6,2) -- (7,1);
    \end{tikzpicture}
    \caption{The central paths of $w = u\mathbf{\textcolor{blue}{u}}du\mathbf{\textcolor{blue}{udd}}dud\in\mathcal{D}(5, 3)$ in blue.}
    \label{fig:dyckpath}
\end{figure}
\end{example}

\begin{remark}
    To the best of our knowledge, the bijection in Theorem~\ref{thm:main} is novel. However, bijections of a similar flavor have appeared previously in the literature. In~\cite{polyominoes}, Delest and Viennot gave a bijection between $\mathcal{D}(m+1,k+1)$ and \emph{parallelogram polyominoes} of perimeter $2m+2$ and width $k$. As shown by Aval, Bergeron, and Garsia in~\cite{AvalBergeron}, parallelogram polyominoes index a linear basis of a ring of $\mathrm{SL}_2(\C)$-invariants of the Grassmannian. We are not aware of a nice bijection between parallelogram polyominoes and plane partitions that would link our bijection to that of~\cite{polyominoes}, nor are we aware of a representation-theoretic connection between the ring of invariants discussed in Section~\ref{se:comtoquant} and the invariant ring considered in~\cite{AvalBergeron}, although either would be quite interesting.
\end{remark}

\section{From Commutative Algebra to Quantum Chemistry}
\label{se:comtoquant}
In this section, we explain the significance of Theorem~\ref{thm:degofex} in quantum chemistry.
Consider an electronic system with $m$ spatial orbitals. 
We model its spin orbitals algebraically using the vector space $\mathcal{Z}_m := \C^m\otimes \C^2$, with explicit basis vectors
\[e_{i} \otimes e_{\downarrow}, \quad e_{i} \otimes e_{\uparrow}, \quad 1 \le i \le m.\]
The element $e_{i\alpha} := e_i\otimes e_\alpha$ represents a spin orbital in position $i$ with spin $\alpha\in\{\downarrow, \uparrow\}$.
We assume our electronic system has $d$ electrons (later, we will assume the electrons come in $k$ pairs and set $d = 2k$). 
The space of \emph{quantum states} for our electronic system is defined as the exterior power $\wedge_{d} \mathcal{Z}_m$ and denoted by $\mathcal{H}_{m, d}$.
A vector space basis for $\mathcal{H}_{m, d}$ is given by the $\binom{2m}{d}$ exterior products of $d$ distinct basis elements $e_{i\alpha}$. 
In \cite{FaulstichSturmfelsSverrisdottir, Sverrisdottir} Faulstich, Sturmfels, and Sverrisd\'ottir use this model to develop the algebraic geometry for electronic structure theory over $\mathcal{H}_{m, d}$. 
However, most physical states satisfy an additional \emph{spin symmetry}: they are invariant under an $\operatorname{SU}(2)$-action, revealing deeper algebraic structures.

\begin{remark}[Lie algebra actions]\label{rem:lie}
    Let $G$ be a connected and simply connected Lie group with corresponding Lie algebra $\mathfrak{g}$.
    The derivative of a $G$-action $\Psi\colon G\to \mathrm{GL}(V)$ on a vector space $V$ is a $\mathfrak{g}$-action $\Psi'\colon \mathfrak{g} \to \operatorname{End}(V)$ on $V$. 
    Derivatives preserve the algebraic properties of the $G$-action -- the $\mathfrak{g}$-representations and $G$-representations are in one-to-one correspondence, see \cite[Section~8.1]{FultonHarris}.
    In particular (and most importantly for our purposes), their trivial representations agree, i.e., the subspace of $V$ that is invariant under the $G$-action is equal to the subspace of $V$ that vanishes under the $\mathfrak{g}$-action.
    Since Lie algebra actions are simpler than Lie group actions (for example, the Lie algebra $\mathfrak{g}$ does not preserve the topological properties of $G$), we use the Lie algebra action to study the invariant space $V^G$. 
\end{remark}

By Remark~\ref{rem:lie}, we may describe the $\operatorname{SU}(2)$-invariant space of $\mathcal{H}_{m,d}$ by finding the vanishing space of $\mathcal{H}_{m,d}$ under the corresponding Lie algebra action.

\begin{definition}
    The \emph{special unitary algebra}, $\mathfrak{su}(2)$, is the Lie algebra of $\operatorname{SU}(2)$.
    It~is~defined~as
\[
\mathfrak{su}(2)=\{A\in\mathbb{C}^{2\times 2}:\ A^\ast=-A,\ {\rm tr}(A)=0\}.
\]
Its complexification $\C \otimes \mathfrak{su}(2)$ is isomorphic to the \textit{special linear algebra} $\mathfrak{sl}_2(\C)$,
which is
generated by the following $2 \times 2$ matrices:
\[
S_+ = \begin{pmatrix}
    0 & 1\\
    0 & 0\\
\end{pmatrix}, \quad S_- = \begin{pmatrix}
    0 & 0\\
    1 & 0\\
\end{pmatrix}, \quad S_z = \begin{pmatrix}
    1 & 0 \\
    0 & - 1\\
\end{pmatrix}.
\]
Since $\mathfrak{sl}_2(\mathbb{C})$ is a simple Lie algebra, the complex vector space representations of $\mathfrak{su}(2)$ are the same as those for $\mathfrak{sl}_2(\mathbb{C})$, see
\cite[p.~439]{FultonHarris}. In particular, by Remark \ref{rem:lie}, the $\mathrm{SU}(2)$- and
$\mathrm{SL}_2(\C)$-invariant subspaces coincide. We will henceforth work
with $\mathfrak{sl}_2(\C)$.
\end{definition}

The $\mathfrak{sl}_2(\C)$-action on $\mathcal{H}_{m,d}$ is canonically obtained from the natural $\mathfrak{sl}_2(\C)$-action on the \textit{spin space} $\C^2$. We explicitly write the basis vectors of the spin space as $e_{\downarrow} = -e_2$ and $e_{\uparrow} = e_1$, where $e_1$ and $e_2$ form the standard basis of $\C^2$. The $\mathfrak{sl}_2(\C)$-action on $\C^2$ is then defined by left multiplication; that is, for $g \in \mathfrak{sl}_2(\C)$ and $v\in\C^2$ the action of $g$ is the map $v \mapsto gv$.
We can finally define the $\mathfrak{sl}_2(\C)$-action on $\mathcal{H}_{m,d}$ by describing the action of $g \in \mathfrak{sl}_2(\C)$ on the basis vectors $e_{I}$, where $I \subseteq [m] \times \{\downarrow, \uparrow\}$:
\begin{equation}\label{eq:wedgeaction}
    g(e_I) = \sum_{j = 1}^d e_{i_1\alpha_1} \wedge \cdots \wedge g(e_{i_j \alpha_j}) \wedge \cdots \wedge e_{i_d\alpha_d} = \sum_{j = 1}^d e_{i_1\alpha_1} \wedge \cdots \wedge (e_{i_j} \otimes ge_{\alpha_j}) \wedge \cdots \wedge e_{i_d\alpha_d}.
\end{equation}
We say a quantum state is \emph{spin adapted} if it vanishes under (\ref{eq:wedgeaction}).  
Our goal is to give an explicit combinatorial model (i.e., an explicit vector space basis) for the space of spin adapted quantum states $\mathcal{H}_{m, d}^{\operatorname{SU}(2)}$.

\begin{remark}\label{rem:pair}
    The operator $S_z$ counts the total spin of the quantum state represented by each basis vector of $\mathcal{H}_{m,d}$. More specifically, for $e_{I\downarrow}\wedge e_{J\uparrow} = e_{i_1\downarrow} \wedge \cdots \wedge e_{i_p\downarrow} \wedge e_{j_1 \uparrow} \wedge \cdots \wedge e_{j_q \uparrow}$,
    $$
    S_z(e_{I\downarrow}\wedge e_{J\uparrow}) = (|J| - |I|)(e_{I\downarrow}\wedge e_{J\uparrow}), 
    $$
    with the total spin of $e_{I\downarrow}\wedge e_{J\uparrow}$  defined as the scalar $|J| - |I|$.
    The kernel of $S_z$ is generated by $ e_{I \downarrow} \wedge e_{J \uparrow}$ where $|I| = |J|$. Hence the vectors in $\mathcal{H}_{m,d}^{\operatorname{SU}(2)}\subseteq\ker(S_z)$ are supported only on the basis vectors with total spin zero.  
\end{remark}
 Remark~\ref{rem:pair} implies that the dimension of the invariant space $\mathcal{H}_{m, d}^{\operatorname{SU}(2)}$ is $0$ whenever $d$ is odd. We henceforth assume $d$ is even and write $d = 2k$ (i.e., we assume our system has $k$ spin-up and $k$ spin-down electrons). In this case, \cite[Theorem 3.6]{FaulstichSverrisdottir} states that 
\begin{equation}\label{eq:invdim}
    \dim_\C\left(\mathcal{H}_{m, 2k}^{\operatorname{SU}(2)}\right) = N(m + 1, k + 1).
\end{equation}

To construct our model, we first observe that the quantum states in $\mathcal{H}_{m, 2k}$ admit an alternative description via 
operators applied to the \textit{reference state} $e_{1, \downarrow} \wedge e_{1, \uparrow} \wedge \cdots \wedge e_{k, \downarrow} \wedge e_{k, \uparrow}$.
Define the \emph{creation operators} $a_{b\alpha}^\dagger$ and \emph{annihilation operators} $a_{i\alpha}$ on the full exterior algebra $\bigwedge\mathcal{Z}_{m}$ for $1\leq i \le k < b\leq m$ and $\alpha \in \{\downarrow, \uparrow\}$ as the following endomorphisms:
\[a_{b\alpha}^\dagger\colon  \bigwedge\mathcal{Z}_m \to \bigwedge \mathcal{Z}_m, \quad \psi \mapsto  e_{b\alpha} \wedge \psi, \quad\quad a_{i\alpha}\colon  \bigwedge\mathcal{Z}_m \to \bigwedge \mathcal{Z}_m, \quad \psi \mapsto e_{i\alpha} \,\lrcorner\, \psi.\]
Here, $\lrcorner$ is the dual operation to the wedge product $\wedge$, see \cite[Section~3.6]{GallierQuaintance} for further explanation. 
The names of these endomorphisms reflect their physical interpretation as creating and annihilating particles on given orbitals.
One can check that they anti-commute, that is,
\begin{equation}\label{eq:cr-an-rel}
    a_{i\alpha}a_{j\beta} + a_{j\beta}a_{i\alpha} = a_{b\alpha}^\dagger a_{c\beta}^\dagger + a_{c\beta}^\dagger a_{b\alpha}^\dagger = a_{b \alpha}^\dagger a_{i\beta} + a_{i \beta}a_{b \alpha}^\dagger = 0.
\end{equation}
These operators generate the exterior algebra $\mathcal{S}$ whose elements are endomorphisms of $\bigwedge\mathcal{Z}_{m}$. 

\begin{remark}
    It is common to define creation and annihilation operators for all indices $1 \le i,b \le m$. In addition to the relations from \eqref{eq:cr-an-rel} they satisfy the following anticommutation relations
    $$
    a_{b \alpha}^\dagger a_{i\beta} + a_{i \beta}a_{b \alpha}^\dagger = \delta_{i,b}\delta_{\alpha,\beta},
    $$
    and generate the Fermi-Dirac algebra (also known as the CAR algebra), see \cite[Section~2]{Sverrisdottir}.
    Since we restrict to creation and annihilation operators on disjoint indexing sets, we work with the subalgebra $\mathcal{S}$ of the Fermi-Dirac algebra, isomorphic to an exterior algebra. See also \cite[Remark~2.4]{Sverrisdottir}.
\end{remark}

Since we wish to study the space of quantum states $\mathcal{H}_{m, 2k}$ (i.e., the $2k$-th graded component of $\bigwedge\mathcal{Z}_m$), we restrict our attention to operators that preserve the total number of electrons. These operators form a subalgebra $R_{m,2k} \subseteq \mathcal{S}$ of endomorphisms on $\mathcal{H}_{m, 2k}$, 
generated by all products $a_{ b\alpha}^\dagger a_{i\beta}$ where $1 \le i \le k < b \le m$ and $\alpha, \beta \in \{\downarrow, \uparrow\}$. This $R_{m, 2k}$ is often called the \emph{number-preserving subalgebra} of $\mathcal{S}$.

The generators of $R_{m, 2k}$ are degree two products of anti-commutative operators, hence they commute. Therefore $R_{m, 2k}$ is a commutative ring, where the ring product is ordinary composition of operators. Note that it is also an $\operatorname{SU}(2)$-module, with the canonical action on $R_{m,2k}$ induced by its realization as a space of endomorphisms on $\mathcal{H}_{m,2k}$.

\begin{prop}[{\cite[Proposition 4.2]{FaulstichSverrisdottir}}]\label{prop:inviso}
    The space of quantum states $\mathcal{H}_{m,2k}$ is isomorphic to the ring $R_{m,2k}$ as an $\operatorname{SU}(2)$-module. In particular $R_{m, 2k}^{\operatorname{SU}(2)}\cong\mathcal{H}_{m, 2k}^{\operatorname{SU}(2)}$.
\end{prop}

\begin{proof}
    Consider the evaluation map of the reference state:
    \[\phi\colon R_{m,2k} \to \mathcal{H}_{m,2k}, \quad T \mapsto T (e_{[k] \downarrow} \wedge e_{[k] \uparrow}).\]
    Since the reference state is $\mathrm{SU}(2)$-invariant, $\phi$ is an $\mathrm{SU}(2)$-equivariant map. It is surjective since every exterior
basis vector is obtained, up to sign, by applying an even product of
annihilation and creation operators to the reference state. If $T$ lies in the kernel of $\phi$, then
commutativity implies that it annihilates every vector in the image of $\phi$.
Surjectivity of $\phi$ therefore implies $T=0$, proving that $\phi$ is an isomorphism of $\mathrm{SU}(2)$-modules.
\end{proof}

By Proposition~\ref{prop:inviso}, a vector space basis for $R_{m, 2k}^{\operatorname{SU}(2)}$ will serve as our desired model for $\mathcal{H}_{m, 2k}^{\operatorname{SU}(2)}$. 
To that end we define the \textit{excitation operators} as elements in $R_{m, 2k}$:
\begin{equation}\label{eq:excitationops}
    \hat X_{i,j} = a_{(j+k) \downarrow}^\dagger a_{i \downarrow} + a_{(j+k) \uparrow}^\dagger a_{i \uparrow}, \quad i\in[k] \text{ and } j\in[m-k].
\end{equation}
They vanish under the $\mathfrak{sl}_2(\C)$-action on $R_{m,2k}$ and are therefore elements of $R_{m, 2k}^{\mathrm{SU}(2)}$.

\begin{theorem}\label{thm:noncompar}
    The invariant ring $R_{m,2k}^{\mathrm{SU}(2)}$ is generated by the excitation operators:
    $$
    R_{m,2k}^{\mathrm{SU}(2)} = \mathbb{C}[\hat X_{i,j} \, :\, 1 \le i \le k, \, 1 \le j \le m - k].
    $$
\end{theorem}

\begin{proof}
    We can write 
    $$
    \mathcal{S} \cong \wedge (V^{\oplus k} \oplus (V^\ast)^{\oplus m - k}) \cong \bigoplus_{p,q} (\wedge^p V^{\oplus k} \otimes \wedge^q (V^\ast)^{\oplus m - k})
    $$
    where $V \cong V_i = \mathrm{span}\{a_{i\downarrow}, a_{i \uparrow}\}$ is the standard representation of $\mathrm{GL}_2(\mathbb{C})$. The number-preserving subalgebra $R_{m,2k}$ consists precisely of
the summands with $p=q$.
    By the extension of the first fundamental theorem of invariant theory to exterior algebras, \cite[Theorem~2]{Howe}, we get that $\mathcal{S}^{\mathrm{GL}_2(\mathbb{C})}$ is generated by the degree two excitation operators \eqref{eq:excitationops}. Furthermore,
since scalar matrices $tI$ act on the $(p,q)$-summand by $t^{p-q}$,
the $\mathrm{GL}_2(\mathbb{C})$-invariants are precisely the
$\mathrm{SL}_2(\mathbb{C})$-invariants with $p=q$. In symbols, $\mathcal{S}^{\mathrm{GL}_2(\mathbb{C})} = R_{m,2k}^{\mathrm{SL}_2(\mathbb{C})}$.
Finally, the $\mathrm{SU}(2)$-invariants coincide with the
$\mathrm{SL}_2(\mathbb{C})$-invariants by complexification, proving the~claim.
\end{proof}

The following corollary of Theorems~\ref{thm:degofex} and ~\ref{thm:noncompar} is a fundamental result. It describes an explicit vector space basis for $R_{m,2k}^{\operatorname{SU}(2)}$, giving the desired combinatorial model for $\mathcal{H}_{m,2k}^{\operatorname{SU}(2)}$.

\begin{corollary}
    The excitation ring $S_{m,k}$ is isomorphic to the $\mathrm{SU}(2)$-invariant ring of $R_{m,2k}$. The monomials
    \[\hat X^M =  \prod_{(i, j)\in[k]\times[m-k]}\hat{X}_{i, j}^{M_{i, j}}, \,\, M \in {\sf Mat}_{k, m-k}(\Z_{\geq 0}) \text{ with } {\sf width}(M)\leq 2, \]
    form a vector space basis for $R_{m,2k}^{\operatorname{SU}(2)}$, indexed by $k \times (m - k) \times 2$ plane partitions.
\end{corollary}

\begin{proof}
Define the map 
\[\psi\colon S_{m,k} \to R_{m,2k}^{\mathrm{SU}(2)}, \qquad X_{i,j} \mapsto \hat X_{i,j}.\]
One can check that the excitation operators fulfill the degree-three relations (\ref{eq:gens}), and hence $\psi$ is well-defined. By Theorem \ref{thm:noncompar}, $\psi$ is surjective. 
By Theorem~\ref{thm:degofex} and equation (\ref{eq:invdim}), the rings also have the same dimension as vector spaces, proving that $\psi$ is an isomorphism. 
The explicit basis for $R_{m, 2k}^{\operatorname{SU}(2)}$ follows by applying $\psi$ to the standard monomial basis for $S_{m, k}$ from Corollary~\ref{cor:stdmono}.
Finally, the bijection between this basis and $k\times (m-k)\times 2$ plane partitions is given by composing the RSK bijections of Theorems~\ref{thm:RSK} and ~\ref{thm:hopkinsRSK}.
\end{proof}

\begin{remark}
    By Proposition \ref{prop:inviso}, the standard monomials of $S_{m,k}$ also give rise to a basis $\{e_M\}$ of $\mathcal{H}_{m,2k}^{\mathrm{SU}(2)}$, where
$$
e_M = \hat X^M (e_{[k] \downarrow} \wedge e_{[k] \uparrow}), \quad M \in {\sf Mat}_{k, m-k}(\Z_{\geq 0}) \,\,\text{with}\,\, {\sf width}(M)\leq 2.
$$
This basis can also be enumerated by $k \times (m - k) \times 2$ plane partitions via RSK.
\end{remark}

We conclude this article by exhibiting a symmetry commonly appearing in quantum chemistry. This symmetry appears at every level of our reasoning, from quantum chemistry through algebra, and into combinatorics.

\begin{remark}[Particle-hole symmetry]
    For any $m,k$ with $k\leq m$, the excitation rings $S_{m,k}$ and $S_{m,m-k}$ have the same dimension, since
    \[
    N(m+1,k+1) \! = \!\frac{1}{m+1} \binom{m+1}{k+1}\!\binom{m+1}{k} \! = \! \frac{1}{m+1} \binom{m+1}{m-k}\!\binom{m+1}{m+1-k}\! = \! N(m+1,m-k+1).
    \]
    In quantum chemistry there exists a classical bijection between $\mathcal{H}_{m,2k}^{\operatorname{SU}(2)}$ and $\mathcal{H}_{m, 2m - 2k}^{\operatorname{SU}(2)}$ known as the \textit{particle-hole symmetry}. It is defined by the map $e_{J \uparrow} \wedge e_{K \downarrow} \mapsto e_{[m] \backslash J \uparrow} \wedge e_{[m] \backslash K \downarrow}$, sending quantum states to their duals. Since $S_{m,k}$ is isomorphic to the invariant space $\mathcal{H}_{m,2k}^{\operatorname{SU}(2)}$ the map lifts to a bijection between the standard monomials of $S_{m,k}$ and $S_{m, m - k}$. 
    Similarly, the map ${\sf Mat}_{k, m-k}(\Z_{\geq 0})\to {\sf Mat}_{m-k, k}(\Z_{\geq 0})$ sending each matrix $M$ to its transpose $M^T$ restricts to a bijection between the standard monomials of $S_{m,k}$ and those of $S_{m,m-k}$.
    The plane partition $B'$ corresponding to $M^T$ under the two RSK correspondences of Theorem~\ref{thm:RSK} and Theorem~\ref{thm:hopkinsRSK} is the transpose of the plane partition $B$ corresponding to $M$. 
    These bijections are compatible in the sense that the following diagram commutes:
    $$
    \begin{tikzcd}[column sep = small, row sep=large]
    \mathcal{B}(k,m-k,2) 
    \arrow[r, "\sim", "RSK"']
    \arrow[d, "B \mapsto B^T", "\cong"']
    &
    \{M\in{\sf Mat}_{k, m-k}(\Z_{\geq 0}) \,:\, {\sf width}(M)\leq 2\}
        \arrow[r, hookrightarrow]
        \arrow[d, "M \mapsto M^T", "\cong"']
    & 
    S_{m,k}
        \arrow[r, "\sim"]
    &
    \mathcal{H}_{m,2k}^{\operatorname{SU}(2)}
        \arrow[d, "\mathrm{PH}", "\cong"']
    \\
    \mathcal{B}(m-k,k,2) 
    \arrow[r, "\sim", "RSK"']
    &
    \{M\in{\sf Mat}_{m-k,k}(\Z_{\geq 0}) \,:\, {\sf width}(M)\leq 2\}
        \arrow[r, hookrightarrow]
    &
    S_{m,m-k}
        \arrow[r, "\sim"]
    &
    \mathcal{H}_{m,2m-2k}^{\operatorname{SU}(2)}
\end{tikzcd}
    $$
    Here the injections map matrices to their corresponding standard monomials.
    We are not aware of a similarly nice bijection between $\mathcal{D}(m+1,k+1)$ and $\mathcal{D}(m+1,m-k+1)$.
\end{remark}

\section*{Acknowledgements}
We thank Bernd Sturmfels for providing feedback on an early draft of this paper. 
We are grateful to the anonymous referees, whose comments greatly improved our proofs and  exposition.
AP and AS were partially supported by an NSF RTG in Combinatorics (DMS 1937241) during the preparation of this material. AS was also supported by an NSF graduate fellowship under Grant No. DGE 21-46756.

\end{document}